\documentclass[12pt,reqno,a4paper]{article}

\usepackage{amsmath,amsthm,amssymb}
\usepackage{authblk}

\usepackage{color}
\definecolor{dred}{RGB}{180,0,0}

\usepackage{enumitem}

\usepackage[T1]{fontenc}
\usepackage{graphics}
\usepackage{fullpage}
\usepackage[colorlinks,linkcolor=blue,citecolor=dred]{hyperref}
\usepackage{mathtools}
\usepackage{multicol}

\usepackage[colorinlistoftodos,prependcaption]{todonotes}
\usepackage[all]{xy}

\makeatletter
\DeclareRobustCommand{\em}{%
  \@nomath\em \if b\expandafter\@car\f@series\@nil
  \normalfont \else \bfseries\itshape \fi}
\makeatother

\newtheorem{theorem}[subsection]{Theorem}
\newtheorem{corollary}[subsection]{Corollary}
\newtheorem{lemma}[subsection]{Lemma}
\newtheorem{proposition}[subsection]{Proposition}

\theoremstyle{definition}
\newtheorem{definition}[subsection]{Definition}

\newtheorem{example}[subsection]{Example}
\theoremstyle{remark}
\newtheorem{remark}[subsection]{Remark}

\numberwithin{equation}{section}
\numberwithin{figure}{section}

\newcommand{\B}[1]{{\mathbf #1}}
\newcommand{\C}[1]{{\mathcal #1}}

\newcommand{\OP}{\operatorname}
\newcommand{\qu}{\triangleright}

\begin{document}

\title{On the geometry and bounded cohomology of racks and quandles}
\author{Jarek K\k{e}dra}
\affil{
University of Aberdeen and University of Szczecin\\
{\tt kedra@abdn.ac.uk}
}



\maketitle

\begin{abstract}
We introduce and investigate a natural family of metrics on connected components
of a rack. The metrics are closely related to certain bi-invariant metrics
on the group of inner automorphisms of the rack. We also introduce a bounded
cohomology of racks and quandles, relate them to the above metrics and prove
a vanishing result for racks and quandles with amenable group of
inner automorphisms.
\end{abstract}

\section{Introduction}\label{S:intro}

Racks and quandles are algebraic structures resembling conjugation in a group
or crossing relations in a knot diagram. They are also invertible solutions of
the Yang-Baxter equation \cite{zbMATH01878448}. Quandles were first observed in
nature by Joyce \cite{zbMATH03744146} and Matveev \cite{zbMATH03828846} and
since then have been spreading through various branches of mathematics
\cite{zbMATH06343943,zbMATH07050797,zbMATH07131234}. Their most successful
applications so far have been in knot theory, where they are used to classify
knots with not so many crossings. A rack is a slightly more general structure
than a quandle, so that every quandle is a rack but not vice versa. Every rack
admits a maximal quandle quotient. 

Although they are usually defined algebraically, as binary operations on a set,
in this paper we propose a more geometric approach. First of all, we mostly
make use of a more geometric definition, which to every element $x$ of a set
$X$ associates a bijection of $X$ plus a certain compatibility axiom
(Definition \ref{D:geo-quandle}).  Secondly, we observe that every rack is
equipped with a family of natural metrics on its connected components
(Definition \ref{D:q-metric}).  In other words, a rack is a disjoint union of
metric spaces and the automorphism group of the rack acts on it by isometries
(Proposition \ref{P:iso}).

Our first result establishes a relation between the rack metric and a certain
bi-invariant metric on the group of inner automorphism of the rack (Theorem
\ref{T:quandle-metric}). As an application of this relatively simple
observation together with the general knowledge of bi-invariant metrics on
groups, we identify families of racks and quandles for which their  metrics on
all connected components have finite diameter. Here is a sample result
(see Corollary \ref{C:chevalley} and \ref{C:Lie}).  

\begin{theorem}\label{T:bounded-racks}
Let $(X,\qu)$ be a rack. If the group of its inner automorphisms is either an
S-arithmetic Chevalley group of higher rank or a semi-simple Lie group with
finite centre then the diameter of every connected component of $(X,\qu)$ is
finite.
\end{theorem}

On the other hand, we show that all connected components of a
free product of quandles (satisfying a mild hypothesis) have infinite diameter
(Proposition \ref{P:free-product}). This holds, for example, for free
nontrivial racks and quandles (Corollary \ref{C:free-qr}).
We also show in Example \ref{E:knot} that the quandle metric on the quandle of
a nontrivial knot has infinite diameter.

In the second part of the paper we introduce the bounded cohomology of racks
and quandles with real coefficients. First, we relate the second bounded
cohomology to the geometry of the above metric (Propositions \ref{P:nontrivial-kernel} and
\ref{P:unb}).

\begin{theorem}\label{T:H2b-rack}
Let $(X,\qu)$ be a rack or a quandle with finitely many connected components.
Then $(X,\qu)$ is unbounded (Definition \ref{D:bounded})
if and only if the comparison map
$H^2_b(X;\B R)\to H^2(X;\B R)$ has nontrivial kernel.
\end{theorem}

Having introduced bounded cohomology, it is natural to test it under a suitable
amenability hypothesis. Functions that are constant on the connected components
of a rack give rise to {\it obvious} nontrivial bounded classes.  If the group
of inner automorphisms of a rack is bounded and amenable then this is, in fact,
all (see Theorem \ref{T:a-bounded} for more precise statement and Remark
\ref{R:a-quandle} for the statement for quandles).

\begin{theorem}\label{T:amenable-rq}
Let $(X,\qu)$ be a rack with bounded, amenable group of inner automorphisms.
Then there is an isomorphism
$$
\OP{Fun}_b(\pi_0(X,\qu)^k)\cong H^k_b(X;\B R),
$$
where $\OP{Fun}_b(\pi_0(X,\qu)^k,\B R)$ denotes the set of bounded functions on
the $k$-fold product of the space $\pi_0(X,\qu)$ of connected components of
the rack.
\end{theorem}

Examples of amenable bounded groups include semisimple compact Lie groups
\cite{zbMATH07757261} and affine Coxeter groups \cite{zbMATH07054643,zbMATH05968646}.
Since $\OP{Fun}_b(\pi_0(X,\qu)^k,\B R)\cong H^k_b(\pi_0(X,\qu))$, where
$\pi_0(X,\qu)$ is a trivial rack (see Example \ref{E:trivial}, \ref{E:pi0}),
the above theorem is indeed a result about the triviality of the bounded
cohomology. The main difference between the above result and the corresponding
one in group theory is the boundedness hypothesis here. The above theorem is 
an extension of a result by Etingof and Gra\~na, who proved an analogous
statement for finite racks \cite[Theorem 4.2]{zbMATH01878448}.

\paragraph{Acknowledgements.} I would like to thank Markus Szymik for
introducing me to racks and quandles and for answering my questions.

\section{Preliminaries}

\subsection{Racks, quandles and their metrics}

\begin{definition}[Algebraic]\label{D:quandle}
A {\em quandle} is a non-empty set $X$ together with a binary operation $\qu\colon X\times X\to X$
satisfying the following axioms:
\begin{enumerate}
\item[{\bf A0}] $x\qu - \colon X\to X$ is a bijection for every $x\in X$;
\item[{\bf A1}] $x\qu (y\qu z) = (x\qu y)\qu(x\qu z)$ for every $x,y,z\in X$;
\item[{\bf A2}] $x\qu x = x$ for every $x\in X$.
\end{enumerate}
If $(X,\qu)$ satisfies the first two axioms only then it is called a {\em rack}.
\end{definition}

\begin{definition}[Geometric]\label{D:geo-quandle}
Let $X$ be a non-empty set and let $\OP{Sym}(X)$ denotes the group of all bijections of $X$.
A {\em quandle} is a map $\psi\colon X\to \OP{Sym}(X)$ such that
\begin{enumerate}
\item[{\bf G1}] $\psi_{\psi_x(y)} = \psi_x\circ\psi_y\circ\psi_x^{-1}$ for all $x,y\in X$;
\item[{\bf G2}] $\psi_x(x)=x$ for all $x\in X$.
\end{enumerate}
If $\psi$ satisfies the first axiom only then it is called a {\em rack}.
The equivalence of the above definitions is given by the formula
$$
\psi_x(y) = x\qu y
$$
and a straightforward verification that suitable axioms are equivalent.
\end{definition}

\begin{definition}\label{D:component}
Let $(X,\qu)$ be a rack. Two elements $x,y\in X$ are equivalent if
there exist $x_1,\ldots,x_n\in X$ such that
$$
y = \left( \psi_{x_1}^{\pm 1}\circ \dots \circ \psi_{x_n}^{\pm 1} \right)(x).
$$
The equivalence class with respect to this equivalence relation is called 
a {\em connected component} of $(X,\qu)$. The set of connected components
of a rack is denoted by $\pi_0(X,\qu)$.
\end{definition}

\begin{definition}\label{D:q-metric}
Let $X_0\subseteq X$ be a connected component of a rack. For $x,y\in X_0$ define
$$
d_0(x,y) = 
\min\left\{ n\in \B N\ |\ y=\left( \psi_{x_1}^{\pm 1}\circ \dots \circ \psi_{x_n}^{\pm 1} \right)(x)\right\}.
$$
It defines a metric on $X_0$ called the {\em rack metric}.
\end{definition}

\begin{remark}
The above metric is the graph metric on the Cayley graph of the rack $(X,\qu)$
associated with the presentation of $(X,\qu)$ given by the multiplication table.
In other words, with presentation of $(X,\qu)$ in which the whole $X$ is a generating set.
\end{remark}

\subsection{Automorphisms of a rack}

A bijective map $\alpha\colon X\to X$ is called an automorphism of a rack
$(X,\qu)$ if
$$
\alpha(x\qu y) = \alpha(x)\qu \alpha(y).
$$
The group of all automorphisms of a rack is denoted by $\OP{Aut}(X,\qu)$.
Observe that for every $x\in X$ the bijection $\psi_x$ is an automorphism.
The subgroup of $\OP{Aut}(X,\qu)$ generated by $\psi_x$ for all $x\in X$
is called the group of inner automorphisms of $(X,\qu)$ and it is denoted
by $G_{\psi}$ or by $\OP{Inn}(X,\qu)$. Notice that a connected component
of a rack is an orbit with respect to the action of the inner automorphism
group. For this reason a connected component is sometimes called an orbit.

\begin{lemma}\label{L:con-inv}
The subset $\psi(X)\subseteq \OP{Aut}(X,\qu)$ is invariant under conjugation.
\end{lemma}
\begin{proof}
Let $x,y\in X$ and let $\alpha\in \OP{Aut}(X,\qu)$. We have that
\begin{align*}
\left(\alpha\circ\psi_x\circ \alpha^{-1}\right)(y) 
&= \alpha\left( \psi_x\left( \alpha^{-1}(y) \right) \right)\\
&= \alpha\left( x\qu \alpha^{-1}(y)\right)\\
&= \alpha(x)\qu y\\
&= \psi_{\alpha(x)}(y),
\end{align*}
which shows that $\alpha\circ\psi_x\circ\alpha^{-1}=\psi_{\alpha(x)}$.
This proves that $\psi(X)$ is invariant under conjugations by all automorphisms.
Notice that this is a strengthening of Axiom G1.
\end{proof}

\begin{proposition}\label{P:iso}
The automorphism group of a rack acts by isometries. More precisely,
if $\alpha\in \OP{Aut}(X,\qu)$ and $x,y\in X_s$ then
$$
d_s(x,y) = d_{\alpha(s)}(\alpha(x),\alpha(y)).
$$
\end{proposition}
\begin{proof}
Let $x,y\in X_s$ be such that $d_s(x,y)=n$. It means that there are $x_1,\ldots,x_n\in X$ such that
$$
y = \left( \psi_{x_1}^{\pm 1}\dots\psi_{x_n}^{\pm 1} \right)(x)
$$
It follows from Lemma \ref{L:con-inv} that $\alpha \circ \psi_z = \psi_{\alpha(z)}\circ \alpha$.
Applying this identity repeatedly yields the following computation.
\begin{align*}
\alpha(y) &= \alpha\left[ \left(\psi_{x_1}^{\pm 1}\dots\psi_{x_n}^{\pm 1} \right)(x)\right]\\
&= \left( \psi_{\alpha(x_1)}^{\pm 1}\dots\psi_{\alpha(x_n)}^{\pm 1} \right)(\alpha(x)),
\end{align*}
which shows that $d_{\alpha(s)}(\alpha(x),\alpha(y))\leq n$. Since $\alpha$ is invertible
we get an equality which finishes the proof.
\end{proof}

\subsection{Representability of racks}

\begin{example}\label{E:}
Let $G$ be a group, $S\subseteq G$ a subset and $\{H_s\leq G\ |\ s\in S\}$ a family
of subgroups such that $H_s\leq Z(s)$, where $Z(s)\leq G$ denotes the centraliser of $s\in S$.
Let $X = \bigsqcup_{s\in S} G/H_s$. The rack operation on $X$ is defined by
$$
xH_s\qu yH_t = xsx^{-1}yH_t.
$$
It is straightforward to verify the axioms. If, moreover, $s\in H_s$ then the above
operation defines a quandle. The rack $(X,\qu)$ defined above 
is denoted by $(G,S,\{H_s\})$.
\hfill $\diamondsuit$
\end{example}

\begin{lemma}[Joyce {\cite[Theorem 7.2]{zbMATH03744146}}]\label{L:joyce}
Every rack $(X,\qu)$ is isomorphic to a rack of the form $(G,S,\{H_s\})$.
\qed
\end{lemma}

\begin{remark}
Joyce proved his theorem for quandles only. The proof for racks is analogous.
\end{remark}

It follows from the above definition that $G$ acts on $(X,\qu)=(G,S,\{H_s\})$ 
by automorphisms.
Moreover, the map $\psi\colon X\to \OP{Aut}(X,\qu)$ factors through $G$
as
$$
\psi_{gH_s} = gsg^{-1}.
$$
Consequently, the group of inner automorphisms $G_{\psi}$ is a subgroup of $G$ 
normally generated by the subset $S$.

\begin{lemma}\label{L:inner<G}
The group $G_{\psi}$ of inner automorphisms of a rack $(G,S,\{H_s\})$ is
normally generated by the subset $S$. In particular, if $S$ normally generates
$G$ then $G=G_{\psi}$.
\qed
\end{lemma}

\subsection{The enveloping group of a rack}

\begin{definition}\label{D:env}
Let $(X,\qu)$ be a rack. The group $G_X$ defined by the presentation
$$
G_X = \langle X \ |\ x\qu y = xyx^{-1},\ x,y \in X\rangle
$$
is called the {\em enveloping group} of the rack $(X,\qu)$.
Notice that  the constant function $X\to \{1\}\subseteq \B Z$ 
defines a surjective homomorphism $G_X\to \B Z$ for
any rack. Consequently, the enveloping group is always infinite.
\end{definition}

\begin{lemma}\label{L:central}
The projection $\pi\colon G_X\to G_{\psi}$ defined by
$\pi(x) = \psi_x$ is a central extension.
If, moreover, the natural map $X\to G_X$ is injective then $\ker\pi = Z(G_X)$.
\end{lemma}
\begin{proof}
The projection $\pi$ is obviously surjective. So we need to check
that its kernel is contained in the centre of $G_X$. Let $g=x_1^{\pm 1} \dots x_n^{\pm 1}\in \ker\pi$ which
means that 
$$
\psi_{x_1}^{\pm 1}\circ \dots\circ \psi_{x_n}^{\pm 1}=\OP{Id}.
$$
Let $x\in X$. Observe that the defining relation can be written as
$$
\psi_{x_i}(x) = x_i x x_i^{-1}.
$$
It follows that
\begin{align*}
gxg^{-1} 
&= \left( x_1^{\pm 1} \dots x_n^{\pm 1} \right)\cdot x \cdot \left( x_n^{\mp 1} \dots x_1^{\mp 1} \right)\\
&= \left( \psi_{x_1}^{\pm 1}\circ \dots \circ\psi_{x_n}^{\pm 1} \right)(x) \\
&= x,
\end{align*}
which shows that $g\in Z(G_X)$ since $x\in X$ was an arbitrary generator. 

Let $X\to G_X$ be injective.
Let $g=x_1^{\pm 1}\dots x_n^{\pm 1}\in Z(G_X)$. Then
we have that $\left( \psi_{x_1}^{\pm 1}\circ \dots\circ \psi_{x_n}^{\pm 1}\right)(x)=x$ for
every $x\in X$ which implies that $\left( \psi_{x_1}^{\pm 1}\circ \dots\circ \psi_{x_n}^{\pm 1} \right)=\OP{Id}$
due to the injectivity of $X\to G_X$.
\end{proof}

\begin{remark}
The map $X\to G_X$ is never injective for racks that are not quandles. Indeed,
if $\psi_x(x)=y\neq x$ then we have that
$$
y=x\qu x = xxx^{-1}=x
$$
in $G_X$. It, moreover, follows that $\psi_x = \psi_{\psi_x^n(x)}$ for every $n\in \B Z$.
\end{remark}

\subsection{Conjugation-invariant norms on groups}\label{SS:word-norms}

Let $G$ be a group and let $S\subset G$ be a generating set. The {\em word norm}
associated with $S$ is defined by
$$
\|g\|_S = \min\left\{n\in \B N\ |\ g=s_1^{\pm 1}\dots s_n^{\pm 1},\ s_i\in S\right\}.
$$
If the subset $S\subseteq G$ is invariant under conjugations then the
norm is {\em conjugation-invariant}, that is,
$$
\|hgh^{-1}\|_S = \|g\|_S,
$$
holds for all $g,h\in G$. The {\em associated metric} is defined by $d_S(g,h) =\|g^{-1}h\|_S$.
It is left-invariant and
if the norm is conjugation-invariant then the metric is {\em bi-invariant}. That
is, both left and right multiplications are isometries of the metric. 

A group $G$ is called {\em bounded} \cite{zbMATH05526532} if every bi-invariant
metric on $G$ has finite diameter.  If a group is generated by a union of
finitely many conjugacy classes then its boundedness is equivalent to the
boundedness of the associated word metric \cite{zbMATH07757261}.  Examples of bounded
groups include S-arithmetic Chevalley groups of higher rank
\cite{zbMATH05936046}, semisimple Lie groups with finite centre,
diffeomorphism groups of compact manifolds and many others
\cite{zbMATH07757261}.

Let $G$ be equipped with a bi-invariant word metric associated with a normally generating
set $S$.
Let $H\leq G$ be a subgroup. The {\em quotient metric} on the quotient $G/H$
is defined by the distance in $G$ between the cosets. Equivalently,
\begin{align*}
d_{S}(xH,yH) 
&= \min\{d_S(xh_1,yh_2)\ |\ h_1,h_2\in H\}\\
&= \min\{d_S(x,yh)\ |\ h\in H\}\\
&= \min\{\|x^{-1}yh\|_S\ |\ h\in H\}\\
&= d_{S}(H,x^{-1}yH).
\end{align*}
In particular, $G$ acts on $G/H$ by isometries.

\begin{example}\label{E:dpsi}
It follows from Lemma \ref{L:con-inv} that $\psi(X)$ generates $G_{\psi}$ and
is invariant under conjugations by elements of ${\rm Aut}(X,\qu)$.
Consequently, the associated word norm $\|g\|_{\psi}$ on $G_{\psi}$ is
$\OP{Aut}(X,\qu)$-invariant (in particular, conjugation-invariant).  
\hfill $\diamondsuit$
\end{example}

\section{The geometry of racks}

\subsubsection*{A characterisation of the rack metric}

With the preparations from the previous section the proof of the following
theorem is fairly obvious.

\begin{theorem}\label{T:quandle-metric}
Let $(X,\qu) = (G, S, \{H_s\})$ be a rack such that $S\subseteq G$ 
normally generates~$G$. Let $X_s = G/H_s$ be a connected component of $(X,\qu)$.
The rack metric $d_s$ on $X_s$ is equivalent to the quotient metric $d_S$ on $G/H_s$.
\end{theorem}
\begin{proof}
Let $x\in G$.
\begin{align*}
d(xH_s,H_s)  
&= \min\left\{n\in \B N\ |\ xH_s = \left( \psi_{x_1H_{s_1}}^{\pm 1}\dots \psi_{x_nH_{s_n}}^{\pm 1} \right)(H_s)\right\}\\
&= \min\left\{n\in \B N\ |\ xH_s = x_1{s_1}^{\pm 1}x_1^{-1}\dots x_ns_n^{\pm 1}x_n^{-1}H_s\right\}\\
&= \min\left\{\|g\|_S \ |\ xH_s = gH_s\right\}\\
&= \min\left\{\|xh\|_S \ |\ h\in H\right\}\\
&= d_S(xH_s,H_s)
\end{align*}
Since both metrics are $G$-invariant, the above computation proves the statement.
\end{proof}

\subsubsection*{A canonical rack-quandle extension}

Let $(X,\qu)$ be a rack and let $y=\psi_x(x)$. We have that
$$
\psi_y = \psi_{\psi_x(x)}=\psi_x\cdot\psi_x\cdot\psi_x^{-1}=\psi_x.
$$
It follows that if $y=\psi_x^n(x)$ then $\psi_y=\psi_x$ for any $n\in \B Z$.
Call $x$ and $y$ equivalent if $y = \psi_x^n(x)$ for some $n\in \B Z$. It is
straightforward to verify that it is an equivalence relation. Let
$\underline{X}=X/\!\approx$ be the quotient. The rack product descends
to the quotient. That is,
$$
[x]\qu [y] = [x\qu y]
$$
is well defined and defines a quandle structure on $\underline{X}$.
We call $X\to \underline{X}$ the {\em canonical rack-quandle extension}.
Notice that this extension maps connected components of the rack to 
connected components of the quandle.

\begin{lemma}\label{L:r-ext-q}
The canonical rack-quandle extension is $1$-Lipschitz. That is,
$$
d_{\underline{X}}([x],[y]) \leq d_{X}(x,y),
$$
for every $x,y\in X$. 
\end{lemma}
\begin{proof}
Suppose that $d_X(x,y)=n$. Then there elements exist $x_1,\ldots, x_n\in X$ such that
$y=\left( \psi_{x_1}^{\pm 1}\dots\psi_{x_n}^{\pm 1} \right)(x)$. It follows that
$$
[y] =\left( \psi_{[x_1]}^{\pm}\dots\psi_{[x_n]}^{\pm 1} \right)[x],
$$
which implies that $d_{\underline{X}}([x],[y])\leq n$.
\end{proof}

\subsubsection*{Bounded racks}

\begin{definition}\label{D:bounded}
If the diameter of the rack metric on each connected component of a rack $(X,\qu)$
is finite then $(X,\qu)$ is called {\em bounded}. Otherwise, it is called
{\em unbounded}.
\end{definition}
In the rack language this means that there exists a number $N>0$ such
that for every $x,y\in X_s$ there exist $x_1,\ldots,x_n\in X$ with $n\leq N$
such that
$$
y =x_n\qu (x_{n-1}\qu (\dots (x_1\qu x)\dots)).
$$
It follows from Lemma \ref{L:r-ext-q} that if the rack in the canonical 
rack-quandle extension is bounded then
so is the underlying quandle. Conversely, if the underlying quandle 
is unbounded then so is the rack.

\begin{corollary}\label{C:bounded}
Let $(X,\qu) = (G,S,\{H_s\})$ be a rack. If the group $G$ is bounded then
the rack $(X,\qu)$ is bounded. \qed
\end{corollary}

In what follows we specify the above corollary to classes of groups that are
well known to be bounded.

\begin{corollary}\label{C:chevalley}
Let $(X,\qu) =(\Gamma,S,\{H_s\})$, where $\Gamma$ is an $S$-arithmetic
Chevalley group of rank at least $2$, $S\subseteq \Gamma$ is a set of root
elements normally generating $\Gamma$ (such a set can be chosen finite) and
$H_s \subseteq Z(s)$. Then the rack $(X,\qu)$ is bounded.
\qed
\end{corollary}

\begin{example}\label{E:bounded}
Consider a rack
$$
\left( \OP{SL}(n,\B Z), E_{12}, Z(E_{12}) \right),
$$
where $n\geq 3$ and $E_{12}$ denotes the elementary matrix with entries $e_{ii}=e_{12}=1$ and
zero otherwise. Notice that the above rack is a quandle.
Since the conjugacy class of $E_{12}$ generates $\OP{SL}(n,\B Z)$ the above quandle
is connected. It is bounded, since $\OP{SL}(n,\B Z)$ is bounded for $n\geq 3$.
\hfill $\diamondsuit$
\end{example}

\begin{corollary}\label{C:Lie}
Let $(X,\qu)=(G,S,\{H_s\})$, where $G$ is a semisimple Lie group with finite centre,
$S$ is a normal generating set (which can be chosen finite) and $H_s\leq Z(s)$.
Then the rack metric on each connected component of $(X,\qu)$ has finite diameter.
\end{corollary}

\begin{example}\label{E:inoue}
Let $\C P\subseteq \OP{PSL}(2,\B C)$ be the quandle consisting of all parabolic elements.
It can be represented as $(\OP{PSL}(2,\B C),S,\{Z(s)\})$, where $S\subseteq \OP{PSL}(2,\B C)$
is a set of representatives of conjugacy classes of parabolic elements. 
Inoue and Kabaya \cite{zbMATH06343943} used the cohomology of this quandle to compute
the complex volume of hyperbolic links. It follows from Theorem \ref{T:quandle-metric}
and boundedness of $\OP{PSL}(2,\B C)$ that the quandle $\C P$ is bounded.
\hfill $\diamondsuit$
\end{example}

\begin{example}\label{E:Lie}
Let $G$ be a bounded simple group and let $1\neq g\in G$. Since the conjugacy
class of $g$ generates $G$, due to simplicity, we get that the rack $(G,g, H_g)$,
where $H_g\subseteq Z(g)$,
is bounded. This, for example, holds for simple Lie groups or (the commutator subgroups of)
Higman-Thompson groups \cite{zbMATH06790225,zbMATH07757261}.
\hfill $\diamondsuit$
\end{example}

All the examples above rely of the fact that the group of inner automorphisms
of the rack is bounded. However, for the boundedness of a connected component
$G/H_s$ is is enough that the embedding $H_s\subseteq G$ is {\em coarsely
surjective}. The latter means that there exists a number $N>0$ such that for
every $g\in G$ there exists $h\in H_s$ with $d_S(g,h)\leq N$.

\begin{example}\label{E:Sp}
Let $\B Z\to \widetilde{\OP{Sp}}(2n;\B Z)\to \OP{Sp}(2n;\B Z)$ be a nontrivial
central extension of the integral symplectic group. This extension is unbounded
and the inclusion of the centre is coarsely surjective 
\cite{zbMATH05936046}. Consequently,
every connected component $\widetilde{\OP{Sp}}(2n;\B Z)/H_s$ of a rack $\left(
\widetilde{\OP{Sp}}(2n;\B Z),S,\{H_s\} \right)$ such that $H_s$ contains the
centre of $\widetilde{\OP{Sp}}(2n;\B Z)$ has finite diameter.

\hfill $\diamondsuit$
\end{example}

\subsubsection*{Unbounded racks and quandles}

\begin{example}\label{E:Zrack}
Let $(\B Z,\qu)$ be a rack defined by $k\qu \ell = \ell +1$. It is connected
and isometric to $\B Z$ with the standard metric, hence unbounded.
\hfill $\diamondsuit$
\end{example}

\begin{example}\label{E:F2}
Let $\B F_2 = \langle x,y\rangle$ be the free group on two generators.
Let $(X,\qu) = (\B F_2,\{x,y\},\{Z(x),Z(y)\})$ be a quandle. It is a union
of conjugacy classes of $x$ and of $y$ and each conjugacy class is a connected
component. Since $Z(x) = \langle x\rangle$ is the cyclic subgroup generated by
$x$ its inclusion into $\B F_2$ is not coarsely surjective (for example 
$d_{\{x,y\}}(Z(x),y^n)=n$) and hence the quandle is unbounded. 
\hfill $\diamondsuit$
\end{example}

\begin{proposition}\label{P:Fn}
Let $\B F_n$ be the free group of rank $n\geq 2$ and let
$(X,\qu) = (\B F_n,S,\{H_s\})$, where $S\subseteq \B F_n$ is finite normally generating subset.
Then each connected component $X_s$ of $(X,\qu)$ has infinite diameter.
\end{proposition}
\begin{proof}
It is enough to show that the embedding
the cetraliser $Z(g)$ of any nontrivial element $g\in \B F_n$ cannot be coarsely surjective.
First observe that the centraliser $Z(g)$ is a cyclic subgroup containing $g$.
The inclusion $Z(g)\to \B F_n$ is never coarsely surjective which can be seen as follows.

Let $\pi\colon \B F_n\to \B Z^n$ be the abelianisation. The projection is Lipschitz with constant
$1$, provided $\B Z^n$ is equipped with the word metric associated with the (finite) generating
set $\pi(S)$. This metric is Lipschitz equivalent to the standard word metric. 
Let $s_1,s_2\in S$ be two generators such that their images $\pi(s_1)$ and $\pi(s_2)$ generate
a free abelian subgroup of rank $2$.
The image $\pi(Z(g))\leq \B Z^n$ is cyclic and hence the distance of either $\pi(s_1^k)$ or $\pi(s_2^k)$
from $\pi(Z(g))$ grows linearly with $k\in \B N$. Since the projection is Lipschitz the same is
true for the distance between $Z(g)$ and $s_1$ or $s_2$ in $\B F_n$.
\end{proof}

The following proposition deals with free products of quandles. See
\cite[Section 7]{zbMATH07227846} for a definition of a free product of quandles
as well as their presentations. Notice, for example, that the free quandle on
$n$ generators is the free product of $n$ copies of the trivial quandle.

\begin{proposition}\label{P:free-product}
Let $(X_i,\qu)$ for $i=1,2$ be quandles with finitely many connected components and
such that the maps $X_i\to G_{X_i}$ are injective.
Let $A_i\subseteq X_i$ be the set of representatives of connected components of $X_i$.
Then every connected component of the free product $X_1* X_2$ has infinite diameter. 
\end{proposition}
\begin{proof}

Recall that the enveloping group $G_{X_i}$ is infinite (see Definition
\ref{D:env}).  It follows from \cite[Theorem 7.2]{zbMATH07227846} that the free
product $X_1*X_2$ is isomorphic to the quandle $(G_{X_1}*G_{X_2},A_1\cup A_2)$;
the latter denotes a quandle defined in \cite[Section 4]{zbMATH07227846}.
Moreover, $G=G_{X_1}*G_{X_2}$ is the enveloping group of $X_1*X_2$ by
\cite[Lemma 7.1]{zbMATH07227846}. Furthermore, Proposition 4.3 of the same
paper implies that this quandle is isomorphic to $(G_{X_1}*G_{X_2},A_1\cup
A_2,\{Z_G(a_1),Z_G(a_2)\})$.  Since both $G_{X_i}$ are infinite, their free
product has trivial centre and we have an isomorphism $G_{X_1}*G_{X_2} \cong
\OP{Inn}(X_1*X_2)$. Thus in order to prove the statement it suffices to show
that the inclusion of the centraliser $Z_G(a_1)$ or $Z_G(a_2)$, where $a_i\in
A_i$, is not coarsely surjective.

To see this, notice that $Z_G(a_i)\leq G_{X_i}\leq G_{X_1}*G_{X_2}$ and
consider the surjective homomorphism $\varphi\colon G_{X_1}*G_{X_2} \to \B F_2
=\langle u_1,u_2\rangle$ defined by $x_i\to u_i$. Notice that this homomorphism
is Lipschitz with respect to the bi-invariant metrics associated with $A_1\cup
A_2$ and $\{u_1,u_2\}$, due to the finiteness of $A_i$.  If the inclusion of
the centraliser $Z_G(a_i)$ was coarsely surjective then the composition with
$\varphi$ would be coarsely surjective. However, $\varphi(Z(a_i))\leq \langle
u_i\rangle$ and hence it is not coarsely surjective. This shows that every
connected component of $X_1* X_2$ has infinite diameter.
\end{proof}

\begin{example}\label{E:free}
Let $T_1=\{x_i\}$ and $T_2=\{x_2\}$ be trivial quandles. 
Each has one connected component represented by $x_i$. 
The enveloping group $G_{T_i}$ of $T_i$ is infinite cyclic $G_{T_i}\cong \B Z$
and hence the free product
$$
T_1*T_2 = (\B F_2=\langle x_1,x_2\rangle, \{x_1,x_2\},\{Z(x_1),Z(x_2)\}).
$$
Notice that this is the quandle considered in Example \ref{E:F2}. It follows
from Proposition~\ref{P:free-product} that both connected components
have infinite diameter.
\hfill $\diamondsuit$
\end{example}

\begin{corollary}\label{C:free-qr}
Every connected component of a free quandle $\OP{FQ}(X)$ has infinite diameter.
Consequently, the same is true for free racks due to Lemma \ref{L:r-ext-q}.
\qed
\end{corollary}

\begin{example}\label{E:knot}
Let $K\subseteq \B S^3$ be a non-trivial knot and let $G_K = \pi_1(\B
S^3\setminus K)$ be the fundamental group of its complement. Let $Q_K$ be the
associated quandle.  Then $Q_K = (G_K,\{s\},P)$, where $s\in G_K$ is the
element represented by the meridian of $K$ and $P$ is the image of
$\pi_1(\partial U)\to G_K$, where $U\subseteq \B S^3$ is a tubular
neighbourhood of $K$ \cite[Corollary 16.2]{zbMATH03744146}.
Notice that $Q_K$ is connected.

Fujiwara \cite[Theorem 1.6]{zbMATH01060878} proved that the second bounded
cohomology of $G_K$ is infinite dimensional.  Since $P$ is abelian, it follows
that there is a non-trivial homogeneous quasi-morphism $q\colon G_K\to \B R$
vanishing on $P$. Let $g\in G_K$ be an element such that $q(g)>0$. Then for any
$h\in P$ we have that
$$
d(g^n,h) = \|g^nh^{-1}\| \geq C(q(gh^{-1})) \geq C(nq(g)-D) 
$$
is arbitrarily large for a large $n\in \B N$. We used here the fact that
quasimorphisms are Lipschitz with respect to conjugation-invariant norms on
normally finitely generated groups and $C>0$ above is the Lipschitz constant;
$D\geq 0$ is the defect of $q$.  This shows that the inclusion $P\subseteq G_K$
is not coarsely surjective and hence the quandle metric has infinite diameter.
\hfill $\diamondsuit$
\end{example}

\section{Bounded cohomology of racks and quandles}\label{S:bcrq}
\begin{remark}
In this paper we consider only the cohomology with real coefficients,
considered as a trivial module. For a general definition of quandle
or rack cohomology see, for example,~\cite{zbMATH01716035,zbMATH01878448}.
\end{remark}

\subsubsection*{Definition of bounded cohomology}
Let $(X,\qu)$ be a quandle. 
Recall that the {\em rack cochain complex} $C^*(X;\B R)$ with real coefficients 
is a complex in which $C^k(X;\B R)$ consist of functions
$f\colon X^k\to \B R$ with
the differential is given by
\begin{align*}
\delta f(x_1,x_2,\ldots,x_{k+1}) &= 
\sum_{i=1}^{k}  (-1)^{i-1}f(x_1,\ldots,x_{i-1},x_{i+1},\ldots,x_{k+1}) \\
&- \sum_{i=1}^{k} (-1)^{i-1}f(x_1,\ldots,x_{i-1},x_i\qu x_{i+1},\ldots,x_i\qu x_{k+1}).
\end{align*}
The cohomology $H^*(X;\B R)$ of the above complex
is called the real {\em rack cohomology} of $(X,\qu)$. 
The subcomplex $C^*_b(X;\B R)\subseteq C^*(X;\B R)$ consisting
of {\it bounded} functions defines the (real) {\em rack bounded cohomology}
of $(X,\qu)$ denoted by $H^*_b(X;\B R)$.
The inclusion $C^*_b(X;\B R)\subseteq C^*(X;\B R)$ induces a homomorphism
$$
H^*_b(X;\B R)\to H^*(X;\B R),
$$
called the {\em comparison} map.
The {\em quandle (bounded) cohomology} is defined as above for the subcomplex
consisting of (bounded) functions $f\colon X^k\to \B R$ satisfying the
additional condition that
\begin{equation}
f(x_1,\ldots,x_k) = 0 \text{ if } x_i=x_{i+1} \text{ for some } i=1,\ldots,k-1.
\label{Eq:q-complex}
\end{equation}
The notation for both the rack and quandle cohomology is for simplicity the
same with the convention that if $(X,\qu)$ is a rack (quandle) then
$H^*_b(X;\B R)$ is the rack (quandle) cohomology.

\begin{remark}
The reason that the quandle cochain complex is smaller is to discard {\it
redundant} cohomology arising from inclusions of singletons which, in the case
of quandles, are retracts.  More precisely, if $x\in X$ then $p\colon X\to
\{x\}$ is a morphism of quandles that has a section $i\colon \{x\}\to X$. The
rack cohomology of the trivial rack $\{x\}$ is isomorphic to $\B R$ in each
degree. Taking into account the retracts for every point of the quandle makes
its rack cohomology unnecessarily enormous; see also \cite[Remark
2.6]{zbMATH07050797}.
\end{remark}

\paragraph{Convention:} {\it In order to make the paper less cumbersome, in
what folllows, we present examples, arguments and proofs mostly for racks. All
arguments carry over almost verbatim for quandles and their bounded cohomology.
If necessary, a separate statement will be given for quandles.  }

\begin{example}\label{E:H1b}
A (bounded) one-cocycle $f\colon X\to \B R$ is constant on each connected
component of $(X,\qu)$.  Indeed, the defining property yields
$$
0 = \delta f(x,y) = f(y) - f(\psi_x(y)).
$$
In particular,
$$
0 = \delta f(x,\psi_x^{-1}(y)) = f(\psi_x^{-1}(y)) - f(y).
$$
This implies that
$f(x) 
= f\left( \left( \psi^{\pm 1}_{x_1}\dots \psi^{\pm 1}_{x_n} \right)(x) \right)$
for every $x,x_1,\ldots,x_n\in X$. Consequently, 
$H^1_b(X;\B R) = \OP{Fun}_b(\pi_0(X,\qu))$ and 
$H^1(X;\B R) = \OP{Fun}(\pi_0(X,\qu))$. 
Notice, that a one-cocycle is a rack homomorphism, where $\B R$ is
considered as a trivial rack.
\hfill $\diamondsuit$
\end{example}

\begin{example}\label{E:trivial}
Let $(X,\qu)$ be the trivial rack. That is $x\qu y=y$ for all
$x,y\in X$ or, equivalently, $\psi\colon X\to \OP{Sym}(X)$ is constant,
equal to the identity. It is immediate to verify that 
the differential in the rack cochain complex is identically zero which
implies that $H^k_b(X;\B R)$ is isomorphic to the space
$\OP{Fun}_b(X^k;\B R)$ of bounded functions on $X^k$. If $(X,\qu)$ is
a quandle then its quandle bounded cohomology of degree $k$
is isomorphic to the space $\underline{\OP{Fun}}_b(X^k,\B R)$ of
bounded functions which are zero on the elements $(x_1,\ldots,x_k)$
such that $x_i=x_{i+1}$ for some $i=1,\ldots,k-1$.
\hfill $\diamondsuit$
\end{example}

\begin{example}\label{E:pi0}
If $(X,\qu)$ is a rack then the set $\pi_0(X,\qu)$ of its connected component
is a trivial rack with respect to the operation induced from $(X,\qu)$. That
is, 
$$
\pi(x)\qu \pi(y) = \pi(x\qu y) = \pi(y).
$$  
The latter also means that the
projection $\pi\colon X\to \pi_0(X,\qu)$ is a morphism of racks.
Hence, it induces a homomorphism
$$
\pi^*\colon H^k_b(\pi_0(X);\B R)=\OP{Fun}_b(\pi_0(X)^k,\B R)\to H^k_b(X;\B R).
$$
There is an analogous homomorphism on quandles defined on
$\underline{\OP{Fun}}_b(X^k,\B R)$.
\hfill $\diamondsuit$
\end{example}

The space of cochains $C^k_b(X;\B R)$ is an $\OP{Aut}(X,\qu)$-module.
Indeed, let $f\colon X^k\to \B R$ be a cochain.
If $\alpha\in \OP{Aut}(X,\qu)$ is an automorphism of $(X,\qu)$ then let
$f\cdot\alpha\colon X^k\to \B R$ be defined by
$$
(f\cdot \alpha)(x_1,\ldots,x_k) = f(\alpha(x_1),\ldots,\alpha(x_k)).
$$
Thus $C^k(X;\B R)$ is a right $\OP{Aut}(X,\qu)$-module.
\begin{lemma}\label{L:Gpsi-mod}
The differential $\delta\colon C^k(X;\B R)\to C^{k+1}(X;\B R)$
is a map of $\OP{Aut}(X,\qu)$-modules.
\end{lemma}
\begin{proof}
We apply the identity $\alpha\circ \psi_x=\psi_{\alpha}\circ \alpha$ 
from Lemma \ref{L:con-inv} in the following computation.
\begin{align*}
(\delta(f\cdot \alpha))(x_1,\ldots,x_{k+1})
&=\sum_{i=1}^k (-1)^{i-1}f(\alpha(x_1),\ldots,\alpha(x_{i-1}),\alpha(x_{i+1}),
\ldots \alpha(x_{k+1}))\\
&-\sum_{i=1}^k (-1)^{i-1} f(\alpha(x_1),\ldots,\alpha(x_{i-1}),\alpha(\psi_{x_i}(x_{i+1})),
\ldots \alpha(\psi_{x_i}(x_{k+1})))\\
&=\sum_{i=1}^k (-1)^{i-1}f(\alpha(x_1),\ldots,\alpha(x_{i-1}),\alpha(x_{i+1}),
\ldots \alpha(x_{k+1}))\\
&-\sum_{i=1}^k (-1)^{i-1} f(\alpha(x_1),\ldots,\alpha(x_{i-1}),\psi_{\alpha(x_i)}(\alpha(x_{i+1}))),
\ldots \psi_{\alpha(x_i)}(\alpha(x_{k+1}))))\\
&=((\delta f)\cdot {\alpha})(x_1,\ldots,x_{k+1}).
\end{align*}
\end{proof}
\begin{corollary}\label{C:}
The $\OP{Aut}(X,\qu)$-invariant cochains form a subcomplex 
$C^*_{b}(X;\B R)^{\rm Aut}$.\qed
\end{corollary}

By restricting the action to the inner automorphism group
$G_{\psi}\leq \OP{Aut}(X,\qu)$ we also obtain an subcomplex
of $G_{\psi}$-invariant cochains and its homology will be
called the {\em invariant bounded} cohomology of $(X,\qu)$.
It will be denoted by $H^*_{b,\rm inv}(X;\B R)$.
The inclusion of the complex induces the homomorphism
$$
H^*_{b,\rm inv}(X;\B R)\to H^*_b(X;\B R).
$$

\begin{lemma}\label{L:f=fz}
If $f\colon X^k\to \B R$ is a cocycle then $f$ and $f\cdot g$ are cohomologous
for every $g\in G_{\psi}$. Consequently, $H^*_b(X;\B R)$ is a trivial
$G_{\psi}$-module.
\end{lemma}
\begin{proof}
Given $z\in X$, let $f_z\in C^{k-1}(X;\B R)$ be defined by 
$$
f_z(x_1,\ldots,x_{k-1}) = f(z,x_1,\ldots,x_{k-1}).
$$
The following equalities are straightforward to verify. The second one follows
because $f$ is a cocycle.
\begin{align*}
\delta (f_{z})(x_1,\ldots,x_k) 
&= (f - f\cdot {\psi_z})(x_1,\ldots,x_k) - \delta f(z,x_1,\ldots,x_k)\\
&= (f-f\cdot{\psi_z})(x_1,\ldots,x_k).
\end{align*}
Since $\psi_z$ generate $G_{\psi}$, we have
\begin{align*}
f-f\cdot g
&= f - f\cdot \psi_{x_1}\dots \psi_{x_n}\\
&= f - f\cdot \psi_{x_1} + f\cdot \psi_{x_1}
- f\cdot \psi_{x_1}\psi_{x_2} + f\cdot \psi_{x_1}\psi_{x_2} -\dots 
-f\cdot \psi_{x_1}\dots \psi_{x_n}\\
&=\delta(f_{x_1}) + \delta\left( (f\cdot\psi_{x_1})_{x_2} \right)
+\dots+ \delta\left( (f\cdot\psi_{x_1}\dots\psi_{x_{n-1}})_{x_n} \right)
\end{align*}
which proves the statement holds for every $g\in G_{\psi}$.
\end{proof}

\section{Rack quasimorphisms and unboundedness}

\begin{definition}\label{D:qqmor}
A {\em rack quasimorphism} is a function $f\colon X\to \B R$ for which
there exists $D\geq 0$ such that
$$
|f(y) - f(x\qu y)| \leq D,
$$
for all $x,y\in X$. If $(X,\qu)$ is a quandle then $f$ will be called
a quandle quasimorphism (definition is the same). Notice that the definition
implies that
$$
\left |f(y) - f(\psi_x^{\pm 1}(y))\right | \leq D,
$$
for all $x,y\in X$.
\end{definition}

\begin{lemma}\label{L:H2b}
If $f\colon X\to \B R$ is a rack quasimorphism then $\delta f$ is a
two-cocycle. The class $[\delta f]\in H^2_b(X;\B R)$ is in the kernel of
the comparison map.
If $f$ is unbounded on a connected component of $(X,\qu)$  then
$\delta f$ is nontrival. 
\end{lemma}
\begin{proof}
The first two statements are obvious. Notice that if $(X,\qu)$ is a quandle
then $\delta f$ is a quandle cocycle, since $\delta f(x,x) = f(x) - f(\psi_x(x))=0$.
Suppose that $f$ is unbounded on a connected component of $(X,\qu)$.  If
$[\delta f]=0$ in $H^2_b(X;\B R)$ then $\delta f=\delta\beta$ for some bounded
cochain $\beta\colon X\to \B R$. It follows that $\delta(f-\beta)=0$, that is
$f-\beta$ is an ordinary one-cocycle and hence it has to be constant on each
connected component of $(X,\qu)$, due to Example \ref{E:H1b}. The latter is
impossible, because $f$ is unbounded and $\beta$ is bounded.
\end{proof}

\begin{example}\label{E:qqm-F2}
Let $\B F_2=\langle x,y\rangle$ be a free group on two generators.
Let $\varphi\colon \B F_2\to \B R$ be a non-trivial homogeneous group quasimorphism.
Let $(X,\qu)= (\B F_2,\{x,y\},\{Z(x),Z(y)\})$ be a free quandle on two generators
and let $\widehat{\varphi}\colon X\to \B R$ be defined by
$$
\widehat{\varphi}(gxg^{-1}) = \varphi(g'),
$$
where $g=g'x^k$ and $g'$ is a reduced word finishing with $y$. Define 
$\widehat{\varphi}(gyg^{-1})$ analogously. 

Since $\varphi$ is a homogeneous quasimorphism on $\B F_2$, it is constant on
conjugacy classes and hence bounded, say by $B>0$, on $C(x^{\pm 1})\cup
C(y^{\pm 1})$.  Let $s\in \B F_2$ be a conjugate of a generator.
We thus have that
\begin{align*}
|\widehat{\varphi}(s\qu gxg^{-1}) - \widehat{\varphi}(gxg^{-1})|
&= |\widehat{\varphi}(sgxg^{-1}s^{-1}) - \varphi(g')|\\
&=|\varphi(sg') - \varphi(g')|\\
&\leq |\varphi(s)| + D\\
&\leq B+D,
\end{align*}
which shows that $\widehat{\varphi}$ is an unbounded quandle quasimorphism
(a similar computation is done for $gyg^{-1}$).
\hfill $\diamondsuit$
\end{example}

\begin{proposition}\label{P:qqm}
A rack quasimorphism $f\colon X\to \B R$ is Lipschitz with respect
to the rack metric. More precisely, there exists a constant $C>0$
such that
$$
|f(x) - f(y)|\leq Cd(x,y)
$$
for any $x,y\in X_s$, where $X_s\subseteq X$ is a connected component.
\end{proposition}
\begin{proof}
Let $x,y\in X_s$ be such that $d(x,y)=n$. It means that 
there exist $x_1,\ldots,x_n\in X$ such that
$y= \left( \psi^{\pm 1}_{x_1}\dots \psi^{\pm 1}_{x_n} \right)(x)$.
By applying inductively the defining property $n$ times we get the following
estimate which proves the statement.
\begin{align*}
|f(y)-&f(x)|
=\left |f\left( \left( \psi^{\pm 1}_{x_1}\dots\psi^{\pm 1}_{x_n} 
\right)(x) \right) -f(x)\right |\\
&=\left |f\left( \left( \psi^{\pm 1}_{x_1}\dots\psi^{\pm 1}_{x_n} \right)(x) \right) 
-f\left( \left( \psi^{\pm 1}_{x_2}\dots\psi^{\pm 1}_{x_n} \right)(x) \right)
+f\left( \left( \psi^{\pm 1}_{x_2}\dots\psi^{\pm 1}_{x_n} \right)(x) \right)
-f(x)\right |\\
&\leq D + \left|f\left( \left( \psi^{\pm 1}_{x_2}\dots\psi^{\pm 1}_{x_n} 
\right)(x) \right) -f(x)\right |\\
&\leq \dots
\leq Dn 
= Dd(x,y).
\end{align*}
\end{proof}

\begin{corollary}\label{C:unb}
If a rack $(X,\qu)$ admits a quasimorphism that is unbounded
on a connected component $X_s$ then this component is unbounded.
\qed
\end{corollary}

\begin{proposition}\label{P:nontrivial-kernel}
If the kernel of the comparison map $H^2_b(X;\B R)\to H^2(X;\B R)$
is nontrivial then the supremum of the diameters of the connected
components of $(X,\qu)$ is infinite. In particular, if $(X,\qu)$
has finitely many connected components then the nontriviality
of the kernel of the above comparison map implies that $(X,\qu)$
is unbounded.
\end{proposition}
\begin{proof}
Suppose that $\sup\{{\rm diam}(X_s,d_s) \ |\ X_s\in \pi_0(X,\qu)\} \leq B<\infty$.
Then, according to Proposition \ref{P:qqm}, for every quasimorphism
$f\colon X\to \B R$ we have that
$$
|f(y)-f(x)|\leq DB,
$$
for every $x,y\in X_s$ and every $X_s\in \pi_0(X,\qu)$. Let $h\colon X\to \B R$
be defined by $h(x) = \min \{f(y)\ |\ y\in X_s\ni x\}$. In other words, $h$
is a function constant on each connected component and equal to the minimum
of $f$ on that component. In particular, it is a $1$-cocycle and the
cochain $f-h$ is bounded and $\delta(f-h) = \delta f$ which means that
the kernel of the comparison map is trivial. The second statement is immediate.
\end{proof}

\begin{proposition}\label{P:unb}
If $(X,\qu)$ is unbounded then the kernel of the comparison map
$H^2_b(X;\B R)\to H^2(X;\B R)$ is nontrivial. In particular,
$H^2_b(X;\B R)\neq 0$.
\end{proposition}
\begin{proof}
Let $x_s\in X_s$ be a basepoint fixed for each connected component
$X_s$ of $(X,\qu)$. Let $f\colon X\to \B R$ be defined by
$$
f(x) = d(x_s,x)
$$
if $x\in X_s$. It is unbounded according to the hypothesis. Then
$$
|\delta f(x,y)| = |f(x) - f(\psi_y(x))| = |d(x_s,x) - d(x_s,\psi_y(x)|\leq 1.
$$
The cocycle is non-zero for a non-trivial rack. Its bounded
cohomology class is non-zero by an argument analogous to the one in
Lemma \ref{L:H2b}.
\end{proof}

\section{Amenability}

This section is motivated by the result in group theory that states that the
bounded cohomology of an amenable group is trivial \cite[Section
3.0]{zbMATH03816552}. We prove that the bounded cohomology of a rack or a
quandle $(X,\qu)$ with amenable inner automorphism group is as trivial as if
$(X,\qu)$ were finite. Our proof follows almost verbatim the proof
of the analogous statement finite racks by Etingof-Gra\~na 
\cite[Theorem 4.2]{zbMATH01878448}.

\begin{definition}\label{D:amenable}
A group $G$ is called {\em amenable} if there exist a functional
${\bf m}\colon\ell^{\infty}(G)\to \B R$ such that
\begin{enumerate}
\item $\B m(1)=1$;
\item if $\varphi\geq 0$ then ${\bf m}(\varphi) \geq 0$;
\item $\B m(\varphi\circ R_h) = \B m(\varphi)$ for every $h\in G$.
\end{enumerate}
Such a functional is called a {\em right-invariant mean} on $G$.
\end{definition}

Let $(X,\qu)$ be a rack and assume that the group $G_{\psi}$ of
its inner automorphism is amenable. 
Let $P\colon C^k_b(X;\B R)\to C^k_{b,\rm inv}(X;\B R)$
be defined by
$$
P(f)(x_1,\ldots,x_k) = \B m(g\mapsto f(g(x_1),\ldots,g(x_k)).
$$

\begin{lemma}\label{L:proj}
Let $(X,\qu)$ be a rack with amenable $G_{\psi}$.
The projection $P$ is a morphism of complexes and it
induces a surjective homomorphism $P\colon H^*_b(X;\B R)\to H^*_{b,\rm inv}(X;\B R)$.
The map $\iota\colon H^*_{b,\rm inv}(X;\B R)\to H^*_b(X;\B R)$
induced by the inclusion of complexes is the right inverse of $P$.
\end{lemma}
\begin{proof}
First observe that $P(f)$ is an invariant cochain.
\begin{align*}
(Pf)(hx_1,\ldots,hx_k)
&= \B m(g\mapsto f(ghx_1,\ldots,ghx_k))\\
&= \B m(g\mapsto f(gx_1,\ldots,gx_k))\\
&= (Pf)(x_1,\ldots,x_k),
\end{align*}
because the mean $\B m$ is right-invariant. 
Secondly, the projection $P$ commutes with differential.
\begin{align*}
P(\delta f)(x_1,\ldots,x_{k+1}) 
&=\B m(g\mapsto \delta f(gx_1,\ldots,gx_{k+1}))\\
&= 
\sum_{i=1}^k (-1)^{i-1}\B m(g\mapsto f(gx_1,\ldots,gx_{i-1},gx_{i+1},\ldots,gx_{k+1})))\\
&-\sum_{i=1}^k (-1)^{i-1}\B m(g\mapsto f(gx_1,\ldots,gx_{i-1},\psi_{gx_i}gx_{i+1},\ldots,
\psi_{gx_i}gx_{k+1})))\\
&= 
\sum_{i=1}^k (-1)^{i-1}\B m(g\mapsto f(gx_1,\ldots,gx_{i-1},gx_{i+1},\ldots,gx_{k+1})))\\
&-\sum_{i=1}^k (-1)^{i-1}\B m(g\mapsto f(gx_1,\ldots,gx_{i-1},g\psi_{x_i}x_{i+1},\ldots,
g\psi_{x_i}x_{k+1})))\\
&=
\sum_{i=1}^k (-1)^{i-1} Pf(x_1,\ldots,x_{i-1},x_{i+1},\ldots,x_{k+1})\\
&-\sum_{i=1}^k (-1)^{i-1}Pf(x_1,\ldots,x_{i-1},\psi_{x_i}x_{i+1},\ldots,\psi_{x_i}x_{k+1})\\
&=\delta (Pf)(x_1,\ldots,x_k).
\end{align*}
The third equality follows from the identity $\psi_{gx}=g\psi_x g^{-1}$.

If $f$ is an invariant cocycle then it is immediate that $P(f)=f$ and 
hence $P\circ \iota=\OP{Id}$ which implies the surjectivity of $P$ and
the last statement.
\end{proof}

\begin{theorem}\label{T:a-bounded}
Let $(X,\qu)$ be a rack with amenable group $G_{\psi}$ of inner automorphisms.
If the conjugation-invariant word norm on $G_{\psi}$ associated with 
$\psi(X)$ has finite diameter then: 
\begin{enumerate}
\item 
the projection 
$$
P\colon H^*_b(X;\B R)\to H^*_{b,\rm inv}(X;\B R)
$$ 
is an isomorphism;
\item
the homomorphism (from Example \ref{E:pi0})
$$
\pi^*\colon\OP{Fun}_b(\pi_0(X)^k,\B R)\to H^k_b(X;\B R)
$$
is an isomorphism.
\end{enumerate}
Consequently we have isomorphisms
$$
\OP{Fun}_b(\pi_0(X)^k,\B R)\cong H^k_{b,\rm inv}(X;\B R)\cong H^k_b(X;\B R).
$$
\end{theorem}
\begin{remark}\label{R:a-quandle}
The statement for quandles is the same except that the
space bounded functions is replaced by
the subspace $\underline{\OP{Fun}}_b(\pi_0(X,\qu)^k,\B R)$
from Example \ref{E:trivial}.
\end{remark}

\begin{proof}
We first prove that the projection 
$P\colon H^*_b(X;\B R)\to H^*_{b,\rm inv}(X;\B R)$
is an isomorphism.
Since the associated word metric on $G_{\psi}$ has finite
diameter, there exists $N\in \B N$ such that for every $g\in G_{\psi}$ we
have that $g=\psi_{x_1}\dots\psi_{x_n}$ for
some $x_1,\ldots,x_n\in G_{\psi}$ and $n\leq N$. Let $f\colon X^k\to \B R$ be
a bounded cocycle. We know from Lemma \ref{L:f=fz} that $f-f\circ \psi_x = \delta(f_x)$.
Applying this identity inductively, we get that
$$
f\cdot g 
= f - \delta\left( (f_{x_1}) + (f\cdot \psi_{x_1})_{x_2} + \dots 
+ (f\cdot \psi_{x_1}\dots \psi_{x_{k-1}})_{x_k} \right).
$$
Let $\alpha_g = (f_{x_1}) + (f\cdot \psi_{x_1})_{x_2} + \dots + (f\cdot \psi_{x_1}\dots \psi_{x_{k-1}})_{x_k}$ 
(notice that there is a choice of a decomposition of $g$ embedded in this definition). Since $f$ is bounded, 
so is the cochain $\alpha_g$, and there exists a bound uniform with respect to $g$.
We claim that $P(f)$ and $f$ are homologous. Indeed,
\begin{align*}
P(f)(x_1,\ldots,x_k) 
&= \B m(g\mapsto (f\cdot g)(x_1,\ldots,x_k))\\
&= \B m(g\mapsto ( (f - \delta\alpha_g)(x_1,\ldots,x_k))\\
&= f(x_1,\ldots,x_k) - \B m(g\mapsto \delta\alpha_g(x_1,\ldots,x_k))\\
&= f(x_1,\ldots,x_k) - \delta \B m(g\mapsto \alpha_g(x_1,\ldots,x_k)).
\end{align*}
Since $\alpha_g$ is uniformly bounded, the cochain $(x_1,\ldots,x_k)\mapsto \B m(g\mapsto \alpha_g(x_1,\ldots,x_k))$ 
is well defined and bounded.
This shows that if $P[f] = 0$ then, since $P(f)$ and $f$ are homologous,
$[f]=0$ which shows the injectivity of $P$ on cohomology. The surjectivity was
shown in Lemma \ref{L:proj}.

Now we prove the second statement. We start with the injectivity of $\pi^*$
and the proof is by induction with respect to the degree. 
The statement is clear in degree zero. Assume it is true for all degrees
smaller than $k$ and let $f\in \OP{Fun}_b(\pi_0(X,\qu)^k,\B R)$ be such
that $\pi^*(f)=\delta \alpha$ for some $\alpha \in C^{k-1}_b(X;\B R)$.
It follows from the first part that we can take $\alpha$ to be an invariant
cochain which implies that
$$
(\pi^*f)_x=(\delta \alpha)_x = - \delta(\alpha_x),
$$
for every $x\in X$. However, $(\pi^*f)_x = \pi^*(f_{[x]})$, where
$f_{[x]}\colon \pi_0(X,\qu)^{k-1}\to \B R$ is defined analogously.
The induction hypothesis implies that $f_{[x]}=0$ for every $x\in X$ 
which proves the statement.

In order to prove surjectivity assume that the mean $\B m$ on
$G_{\psi}$ is bi-invariant. Such a mean always exists on an amenable
group $G$ \cite[Exercise 1.26]{zbMATH00193576} which can be seen by taking
the left-right action of $G\times G$ on the compact convex set of
means on $G$. Since $G\times G$ is amenable as well, this action
has a fixed point.

The second ingredient of the proof is the cochain complex decomposition
$$
C^*_b(X;\B R)=C_{b,\rm inv}^*(X;\B R)\oplus (1-P)C^*_b(X;\B R);
$$
notice that the second summand is acyclic. 

Let $[f]\in H^k_b(X;\B R)$ be any element. We can assume it is represented
by an invariant cocycle $f$. The invariance of $f$ implies 
that $f_x$ is a cocycle for every $x\in X$. Consider the decomposition
$$
f_x = P(f_x) + (1-P)(f_x)
$$ 
for every $x\in X$. Let $f^+,f^-\in C_b^k(X;\B R)$ be cochains such that
$$
f^+_x = P(f_x)\qquad\text{and}\qquad f^-_x = (1-P)(f_x),
$$
for every $x\in X$. The aim is to show that $f^+$ is an invariant cocycle
homologous to $f$. Let's show the invariance first and let's start with
the following computation. Let $h\in G_{\psi}$ and $x\in X$ be any elements.
\begin{align*}
f^+_{hx}(x_2,\ldots,x_k)
&= P(f_{hx})(x_2,\ldots,x_k)\\
&= \B m(g\mapsto f_{hx}(gx_2,\ldots,gx_k))\\
&= \B m(g\mapsto f_{hx}(hgx_2,\ldots,hgx_k)) 
&\text{\tt by left-invariance of $\B m$}\\
&= \B m(g\mapsto f(hx, hgx_2,\ldots,hgx_k))\\ 
&= \B m(g\mapsto f(x, gx_2,\ldots,gx_k)) &\text{\tt by invariance of $f$}\\
&= P(f_x)(x_2,\ldots,x_k)\\ 
&= f^+_x(x_2,\ldots,x_k).
\end{align*}
It follows that
\begin{align*}
(f^+\cdot h)_x &= f^+_{hx}\cdot h &\text{\tt obvious}\\
&= f^+_x\cdot h &\text{\tt by the previous computation}\\
&= P(f_x)\cdot h\\ 
&= P(f_x) &\text{\tt by the invariance of $P(f_x)$}\\ 
&= f^+_x.
\end{align*}
The invariance of $f^+$ is now clear:
\begin{align*}
f^+(hx_1,\ldots,hx_k)
&= f^+_{hx_1}(hx_2,\ldots,hx_k)= f^+_{x_1}(hx_2,\ldots,hx_k)\\
&= (f^+_{x_1}\cdot h)(x_2,\ldots,x_k)
= f^+_{x_1}(x_2,\ldots,x_k) = f^+(x_1,\ldots,x_k).
\end{align*}
Next observe that $f^+$ is a cocycle:
\begin{align*}
\delta f^+(x,x_1,\ldots,x_{k})
&= f^+(x_1,\ldots,x_{k}) - f^+(\psi_{x}x_1,\ldots,\psi_{x}x_{k})
- \delta f^+_x(x_1,\ldots,x_k) =0.
\end{align*}
The sum of the first two terms vanish by invariance of $f^+$ and
the $\delta f^+_x = \delta P(f_x)=0$ because $f_x$ is a cocycle
since $f$ is an invariant one.
We obtain that both summands in the decomposition $f=f^+ + f^-$
are invariant cocycles, since $f$ and $f^+$ are.

Next we show that $f^-$ is a coboundary. Let $\alpha\in C^{k-1}(X;\B R)$
be such that $\delta(\alpha_x) = f^-_x$ for every $x\in X$. We have
that
$$
\delta\left( (\alpha \cdot g)_x \right)
= \delta(\alpha_{gx}\cdot g)
= \delta(\alpha_{gx})\cdot g
= f^-_{gx}\cdot g
= (f^-\cdot g)_x
= f^-_x,
$$
for every $g\in G_{\psi}$ and $x\in X$.
It follows that
\begin{align*}
\delta\left( (P\alpha)_x \right)(x_1,\ldots,x_k)
&= (P\alpha)_x(x_2,\ldots,x_k) 
- (P\alpha)_x(\psi_{x_1}x_2,\ldots,\psi_{x_1}x_k)-\dots\\
&= (P\alpha)(x,x_2,\ldots,x_k) 
- (P\alpha)(x,\psi_{x_1}x_2,\ldots,\psi_{x_1}x_k)-\dots\\
&=\B m(g\mapsto \alpha(gx,gx_2,\ldots,gx_k) 
- \alpha(gx,g\psi_{x_1}x_2,\ldots,g\psi_{x_1}x_k) - \dots)\\ 
&=\B m(g\mapsto (\alpha\cdot g)(x,x_2,\ldots,x_k) 
- (\alpha\cdot g)(x,\psi_{x_1}x_2,\ldots,\psi_{x_1}x_k) - \dots)\\
&=\B m(g\mapsto (\alpha\cdot g)_x(x_2,\ldots,x_k) 
- (\alpha\cdot g)_x(\psi_{x_1}x_2,\ldots,\psi_{x_1}x_k) - \dots)\\
&=\B m(g\mapsto \delta\left( (\alpha\cdot g)_x \right)(x_1,\ldots,x_k))\\
&=\B m(g\mapsto f^-_x(x_1,\ldots,x_k))
=f^-_x(x_1,\ldots,x_k).
\end{align*}
Since $P\alpha$ is invariant we get that
$$
(\delta(P\alpha))_x = \delta((P\alpha)_x) = f^-_x,
$$
for every $x\in X$ which implies that $\delta\alpha = f^-$.
Hence we can assume that the class $[f]$ is represented
by $f^+$.

Since $f^+_x = P(f_x)$ is an invariant $(k-1)$-cocycle for every $x\in X$, 
we can consider $f^+$ as a cocycle of the form
$$
\sum_{s\in \pi_0(X,\qu)}I_{s}\otimes f^+_s 
\in Z^1_b(X;\B R)\otimes Z^{k-1}_{b,\rm inv}(X;\B R)
\subseteq Z_{b,\rm inv}^k(X;\B R),
$$
where $I_s$ is the indicator function of $\{s\}$, that
is $I_{s}(t) = 1$ if $s=t$ and zero otherwise.
Proceeding by induction we get that every cohomology class
in $H^k_b(X;\B R)$ is decomposed as a product of classes
of degree one and hence they are represented by cocycles from
$\OP{Fun}_b(\pi_0(X)^k,\B R)$.
\end{proof}

\bibliography{/home/kedra/sync/bibliography/bibliography}

\begin{thebibliography}{10}

\bibitem{zbMATH07227846}
Valeriy Bardakov and Timur Nasybullov.
\newblock Embeddings of quandles into groups.
\newblock {\em J. Algebra Appl.}, 19(7):20, 2020.
\newblock Id/No 2050136.

\bibitem{zbMATH05526532}
Dmitri Burago, Sergei Ivanov, and Leonid Polterovich.
\newblock Conjugation-invariant norms on groups of geometric origin.
\newblock In {\em Groups of diffeomorphisms in honor of Shigeyuki Morita on the
  occasion of his 60th birthday. Based on the international symposium on groups
  and diffeomorphisms 2006, Tokyo, Japan, September 11--15, 2006}, pages
  221--250. Tokyo: Mathematical Society of Japan, 2008.

\bibitem{zbMATH01716035}
J.~Scott Carter, Mohamed Elhamdadi, and Masahico Saito.
\newblock Twisted quandle homology theory and cocycle knot invariants.
\newblock {\em Algebr. Geom. Topol.}, 2:95--135, 2002.

\bibitem{zbMATH01878448}
P.~Etingof and M.~Gra{\~n}a.
\newblock On rack cohomology.
\newblock {\em J. Pure Appl. Algebra}, 177(1):49--59, 2003.

\bibitem{zbMATH01060878}
Koji Fujiwara.
\newblock The second bounded cohomology of a group acting on a
  {Gromov}-hyperbolic space.
\newblock {\em Proc. Lond. Math. Soc. (3)}, 76(1):1, 1998.

\bibitem{zbMATH06790225}
{\'S}wiatos{\l}aw~R. Gal and Jakub Gismatullin.
\newblock Uniform simplicity of groups with proximal action.
\newblock {\em Trans. Am. Math. Soc., Ser. B}, 4:110--130, 2017.

\bibitem{zbMATH05936046}
{\'S}wiatos{\l}aw~R. Gal and Jarek K{\k{e}}dra.
\newblock On bi-invariant word metrics.
\newblock {\em J. Topol. Anal.}, 3(2):161--175, 2011.

\bibitem{zbMATH03816552}
M.~Gromov.
\newblock Volume and bounded cohomology.
\newblock {\em Publ. Math., Inst. Hautes {\'E}tud. Sci.}, 56:5--99, 1982.

\bibitem{zbMATH06343943}
Ayumu Inoue and Yuichi Kabaya.
\newblock Quandle homology and complex volume.
\newblock {\em Geom. Dedicata}, 171:265--292, 2014.

\bibitem{zbMATH03744146}
David Joyce.
\newblock A classifying invariant of knots, the knot quandle.
\newblock {\em J. Pure Appl. Algebra}, 23:37--65, 1982.

\bibitem{zbMATH07757261}
Jarek K{\k{e}}dra, Assaf Libman, and Ben Martin.
\newblock Strong and uniform boundedness of groups.
\newblock {\em J. Topol. Anal.}, 15(3):707--739, 2023.

\bibitem{zbMATH07054643}
Joel~Brewster Lewis, Jon McCammond, T.~Kyle Petersen, and Petra Schwer.
\newblock Computing reflection length in an affine {Coxeter} group.
\newblock {\em S{\'e}min. Lothar. Comb.}, 80B:12, 2018.
\newblock Id/No 51.

\bibitem{zbMATH03828846}
S.~V. Matveev.
\newblock Distributive groupoids in knot theory.
\newblock {\em Math. USSR, Sb.}, 47:73--83, 1984.

\bibitem{zbMATH05968646}
Jon McCammond and T.~Kyle Petersen.
\newblock Bounding reflection length in an affine {Coxeter} group.
\newblock {\em J. Algebr. Comb.}, 34(4):711--719, 2011.

\bibitem{zbMATH00193576}
Alan L.~T. Paterson.
\newblock {\em Amenability}, volume~29 of {\em Math. Surv. Monogr.}
\newblock Providence, RI: American Mathematical Society, 1988.

\bibitem{zbMATH07050797}
Markus Szymik.
\newblock Quandle cohomology is a {Quillen} cohomology.
\newblock {\em Trans. Am. Math. Soc.}, 371(8):5823--5839, 2019.

\bibitem{zbMATH07131234}
Nobuyoshi Takahashi.
\newblock Quandles associated to {Galois} covers of arithmetic schemes.
\newblock {\em Kyushu J. Math.}, 73(1):145--164, 2019.

\end{thebibliography}
\bibliographystyle{plain}

\end{document}